\documentclass[preprint,12pt]{elsarticle}
\usepackage{amsthm,amsfonts,amssymb,amscd,amsmath,enumerate,verbatim,calc,graphicx,geometry}
\usepackage[all]{xy}
\newtheorem{theorem}{Theorem}[section]
\newtheorem{lemma}[theorem]{Lemma}
\newtheorem{proposition}[theorem]{Proposition}
\newtheorem{corollary}[theorem]{Corollary}
\theoremstyle{definition}
\theoremstyle{definitions}
\newtheorem{definition}[theorem]{Definition}

\newtheorem{example}[theorem]{Example}
\theoremstyle{notations}

\theoremstyle{remarks}

\journal{ }

\begin{document}

\begin{frontmatter}



\title{On Topologized Fundamental Group and covering spaces of topological groups}


\author[]{Hamid~Torabi\corref{cor1}}
\ead{h.torabi@ferdowsi.um.ac.ir}

\address{Department of Pure Mathematics, Ferdowsi University of Mashhad,\\
P.O.Box 1159-91775, Mashhad, Iran.}
\cortext[cor1]{Corresponding author}
\begin{abstract}
  In this paper, we show that every topological group is a strong small loop transfer space at the identity element. This implies that the quasitopological fundamental group of a connected locally path connected toplogical group is a topological group. Also, we show that every covering space $\widetilde{G}$ of a connected locally path connected topological group $G$ is a topological group. Furthermore we prove that the covering map $p :\widetilde{G} \to G$ is homomorphism.
\end{abstract}


\end{frontmatter}

\section{Introduction}
 The fundamental group endowed with the quotient topology induced by the natural surjective map $q : \Omega(X,x_{0})\rightarrow\pi_{1}(X,x_{0}) $, where $ \Omega(X,x_{0}) $ is the loop space of  $(X,x_{0})$ with compact-open topology, denoted by $\pi_{1}^{qtop}(X,x_{0}) $  becomes a quasitopological group (see \cite{2,3}). In fact,  $\pi_{1}^{qtop}(X,x_{0}) $  is not a topological group in general. However, it seems interesting to find out when the quasitopological fundamental group is a topological group. For example, Fabel \cite{10} showed that the quasitopological fundamental group of the Hawaiian earring is not a topological group. On the other hand, torabi et al. \cite{t2} proved that the quasitopological fundamental group of a connected locally path connected, semi locally small generated space is a topological group. In this paper we show that the quasitopological fundamental group of a connected locally path connected toplogical group is a topological group. \\
Spanier \cite[Theorem 13 on page 82]{s} introduced a different topology on the fundamental group which has been called the whisker topology by Brodskiy et al. \cite{BroU} and denoted by $\pi_{1}^{wh}(X,x_{0})$, which is not even quasitopological group, in general. Recal that for any pointed topological space $(X, x_{0})$ the whisker topology on the set $\pi_{1}(X,x_{0})$ is defined by the collection of all the following sets as a basis
 $$[\alpha]i_{*}\pi_{1}(U,x_{0})=\lbrace [\beta] \in \pi_{1}(X,x_{0}) \ \vert \ \beta \simeq \alpha\ast\delta \ for \ some \ loop \ 
  \delta:I\rightarrow U \rbrace,$$
 where $[\alpha] \in \pi_{1}(X,x_{0})$ and $U$ is an open neighborhood of $x_{0}$.\\
  The concept of small loop transfer space which have been introduced and studied by Brodskiy et al. \cite{BroU} is defined bellow.
\begin{definition}
 A topological space $X$ is called a small loop transfer (SLT for short) space at $x_{0}$ if for every path $ \alpha $ in $X$ with $\alpha(0)=x_{0}$  and for every neighborhood $U$ of $x_{0}$ there is a neighborhood $V$ of $\alpha(1)=x$ such that for every loop $\beta$ in $V$ based at $x$ there is a loop $\gamma$ in $U$ based at $x_{0}$  which is homotopic to 
$ \alpha\ast\beta\ast\bar{\alpha}$ relative to $\dot{I}$. The space $X$ is called an SLT space if $X$ is SLT at $x_{0}$ for every $x_{0 }\in{X}$.
\end{definition}
 Brodskiy et al. \cite{BroU} proved that if a peano space $X$ is an SLT space, then $\pi_{1}^{qtop}(X,x_0) = \pi_{1}^{wh}(X,x_0) $. In Section 2, we show that the whisker and the compact-open topologies on $\pi_{1}(G,e_G)$ coincide for a connected and locally path connected topological group $G$.\\
 
    The concepts of \textit{homotopically Hausdorff relative to} $H$ and \textit{homotopically path Hausdorff  relative to} $H$ are introduced in \cite{Zastrow} and \cite{BrazFa}, respectively. Let $H\leq \pi_{1}(X,x_{0})$. Recall that $X$ is said to be homotopically Hausdorff relative to $ H $ if for every $x\in{X}$, for every $ \alpha \in P(X,x_0) $ with $\alpha(1)=x$, and for any $ g \in{\pi_{1}(X,x_{0})}\setminus H$, there is an open neighborhood $ U_g $ of $ \alpha(1) $ in which there is no loop $ \gamma:(I, \dot{I})\rightarrow (U_g,\alpha(1)) $  such that $ [ \alpha\ast\gamma\ast{\alpha}^{-1}]\in Hg$, where $Hg$ is the right coset of $H$ in $\pi_{1}(X,x_{0})$ with respect to $g$ (see \cite[p. 190]{Zastrow}). Note that $X$ is homotopically Hausdorff if and only if $X$ is homotopically Hausdorff relative to the trivial subgroup $H = \lbrace 1\rbrace$. Also, $X$ is said to be homotopically path Hausdorff relative to $H$, if for every pair of paths $ \alpha, \beta \in P(X,x_0) $  with $ \alpha(1)=\beta(1) $  and $ [\alpha\ast{\beta}^{-1}] \notin H $, there is a partition $ 0=t_{0}<t_{1}<...<t_{n}=1 $ and a sequence of open subsets $ U_{1}, U_{2}, ..., U_{n} $ with  $ \alpha([t_{i-1},t_{i}])\subseteq U_{i} $ such that if $ \gamma:I \rightarrow X $  is another path satisfying $  \gamma([t_{i-1},t_{i}])\subseteq U_{i} $ for $1\leq i \leq n$ and $ \gamma(t_{i})=\alpha(t_{i})  $ for every $0\leq i\leq n $, then $ [\gamma \ast{\beta}^{-1}] \notin H (see $ \cite{BrazFa}). In Section 2, we show that for a subgroup $H$ of the fundamental group of topological group $G$, $G$ is homotopically Hausdorff relative to $H$ if and only if $G$ is homotopically path Hausdorff relative to $H$.\\
   
   Chevalley \cite{L} introduced a covering group theory for connected, locally path connected and semi locally simply connected topological groups. Rotman \cite[Theorem 10.42]{r} prove that for every covering space $(\widetilde{X}, p)$ of a connected, locally path connected, semi locally simply connected topological group $G$, $\widetilde{X}$ is a topological group and $p$ is a homomorphism. In section 3, we introduce covering theory for topological group and give a classification for covering groups of connected locally path connected topological groups. Also we show that for every covering space $(\widetilde{X}, p)$ of a connected, locally path connected topological group $G$, $\widetilde{X}$ is a topological group and $p$ is a homomorphism.

 \section{Topologized fundamental group of topological group}
Let $ G$ be a topological group and $\alpha$  be a path in $G$, then we denote the homotopy class $\alpha$ by $[\alpha]$ and the inverse of $\alpha$ by $\bar{\alpha}$ where $\bar{\alpha}:I \to G$ by $\bar{\alpha}(t)=\alpha(1-t)$. Also we define $\alpha^{-1}:I \to G$ by $\alpha^{-1}(t)= (\alpha(t))^{-1}$ and denote the constant path $\alpha:I \to G$ at $a \in G$ by $C_a$.

\begin{definition}\label{path}
Let $ G$ be a topological group with multiplication map $m: G \times G \to G$, given by $(x, y) \to xy$.  Let $\alpha, \beta$ be two paths in $G$. We define the path $ \alpha.\beta:I \to G$ by $ \alpha.\beta(t)=m(\alpha(t),\beta(t)).$
Since the  multiplication map  and $ \alpha ,\beta$ are continuous, $ \alpha.\beta:I \to G$ is continuous. 
\end{definition}

Let $f$  be a path in $G$ and $a \in G$. We denote the path $C_a.f$ and $f.C_a$ by $ ^af $ and $f^a$ respectively.

\begin{lemma}\label{multi}

 If $ G$ is a topological group and $f,g$ be two loops in $G$ based at $a \in G$ and $b \in G$ respectively, then $ [f.g]= [f^b][^ag]$. In particular, if $f,g$ be two loops in $G$ based at the identity element $e_G$, then $ [f.g]= [f][g]$.

\end{lemma}

\begin{proof}
Consider the continuous multiplication map $m: G \times G \to G$, given by $(x, y) \to xy$. Let $\theta: \pi_1(G, a) \times \pi_1(G, b) \to \pi_1(G\times G, (a,b))$ be the isomorphism defined by $([f],[g]) \to [(f,g)]$  (see \cite[Theorem 3.7]{r}). Since $m_*\theta:\pi_1(G, a) \times \pi_1(G, b) \to \pi_1(G, ab)$ is a homomorphism and $([f], [g]) = ([f], [C_b])([C_a], [g])$, we have
$$m_*\theta ([f],[g])=m_*\theta(([f], [C_b])([C_a], [g]))=m_*\theta(([f], [C_b])) m_*\theta(([C_a], [g]))= [f^b][^ag] .$$
On the other hand $m_*\theta ([f],[g])= [f.g]$, which implies that $[f^b][^ag] = [f.g]$
\end{proof}

The concept of strong small loop transfer space which have been introduced and studied by Brodskiy et al. \cite{BroU} is defined as follows.

\begin{definition}
A topological space $X$ is called a strong small loop transfer (strong SLT for short) space at $x_{0}$ if for every $x \in X$ and for every neighborhood $U$ of $x_{0}$ there is a neighborhood $V$ of $x$ such that for every path $ \alpha $ in $X$ with $\alpha(0)=x_{0}, \alpha(1)=x$  and  for every loop $\beta$ in $V$ based at $x$ there is a loop $\gamma$ in $U$ based at $x_{0}$  which is homotopic to 
$ \alpha\ast\beta\ast\bar{\alpha}$ relative to $\dot{I}$. The space $X$ is called a strong SLT space if $X$ is strong SLT at $x_{0}$ for every $x_{0 }\in{X}$.
\end{definition}

\begin{theorem}
A topological group $G$ is a strong SLT space at the identity element $e_G$.
\end{theorem}

\begin{proof}
Let $U$ be a neighborhood of $e_{G}$ in $G$ and $x \in G$. We show that for every loop $\beta$ based at $x$ in the neighborhood $xU = \{ xu | u \in U \}$ of $x$ and every path $ \alpha $ in $G$ with $\alpha(0)=e_{G} , \alpha(1) = x$, there is a loop $\gamma$ in $U$ based at $e_G$  which is homotopic to $ \alpha\ast\beta\ast\bar{\alpha}$ relative to $\dot{I}$. For this let $f$ be a loop in $G$ based at $e_G$ such that
\begin{displaymath}
f(t)= \left\{
\begin{array}{lr}
\alpha(3t)                     &       0\leq t\leq 1/3 \\
x     &      1/3\leq t\leq 2/3\\
\bar{\alpha}(3t-2)              &  2/3\leq t\leq 1.
\end{array}
\right.
\end{displaymath}
Also let $g$ be a loop in $G$ based at $e_G$ such that
\begin{displaymath}
g(t)= \left\{
\begin{array}{lr}
e_G                    &       0\leq t\leq 1/3 \\
x^{-1}\beta(3t-1)     &      1/3\leq t\leq 2/3\\
e_G             &  2/3\leq t\leq 1.
\end{array}
\right.
\end{displaymath}
Therefore by Lemma 2.2 we have $[f][g] = [f.g]$. Note that
\begin{displaymath}
(f.g)(t)= \left\{
\begin{array}{lr}
\alpha(3t)                     &       0\leq t\leq 1/3 \\
\beta(3t-1)     &      1/3\leq t\leq 2/3\\
\bar{\alpha}(3t-2)             &  2/3\leq t\leq 1.
\end{array}
\right.
\end{displaymath}
If $\gamma$ is a loop in $U$ based at $e_G$ such that for every $t \in I$, $\gamma (t) = x^{-1}\beta(t)$ then we have
$$[\gamma] = [\alpha * \bar{\alpha}][\gamma] = [f][g] = [f.g] = [\alpha\ast\beta\ast\bar{\alpha}] $$
Hence $G$ is a strong SLT space at $e_G$.
\end{proof}

It is easy to see that every strong SLT space at $x_{0}$  is an SLT space at $x_{0}$. Therefore we have the following results.

\begin{corollary}
Let $G$ be a topological group. Then $G$ is an SLT space at the identity element $e_G$.
\end{corollary}

\begin{corollary}
Let $G$ be a connected and locally path connected topological group, then $\pi_{1}^{qtop}(G,e_G) = \pi_{1}^{wh}(G,e_G) $ is  a topological group.
\end{corollary}

\begin{proof}
By corollary 2.5, $G$ is an SLT space at $e_G$. Therefore $\pi_{1}^{qtop}(G,e_G) = \pi_{1}^{wh}(G,e_G) $ by \cite[Theorem 2.5]{jamali}. Hence $\pi_{1}^{qtop}(G,e_G)$ and $\pi_{1}^{wh}(G,e_G) $ are topological group by \cite[Corollary 2.6]{jamali}.
\end{proof}

\begin{proposition}
Let $G$ connected and locally path connected topological group and $H \leqslant \pi_{1}(G,e_G)$. Then the following statment are equivalent.

(i) $H$ is an open subgroup of $\pi_{1}^{qtop}(G,e_G)$.\

(ii) $H$ is an open subgroup of $\pi_{1}^{wh}(G,e_G)$.\

(iii) There is a neighborhood $U$ of $e_{G}$ s.t. $i_{*}\pi_{1}(U,e_G)  \leqslant  H $.\

\end{proposition}

\begin{proof}
$(i) \Leftrightarrow (ii)$ follows  from Corollary 2.6.\

$(ii) \Rightarrow (iii):$ Let $H$ be an open subgroup of $\pi_{1}^{wh}(G,e_G)$. Since $i_{*}\pi_{1}(V,e_G)$ is an open basis in $\pi_{1}^{wh}(G,e_G)$, then there is a neighborhood $U$ of $e_{G}$ s.t. $i_{*}\pi_{1}(U,e_G) \leqslant H $.\

$(iii) \Rightarrow (ii):$ Let there is a neighborhood $U$ of $e_{G}$ s.t. $i_{*}\pi_{1}(U,e_G) \leqslant H $. Since $i_{*}\pi_{1}(U,e_G)$ is an open set in $\pi_{1}^{wh}(G,e_G)$ and $i_{*}\pi_{1}(U,e_G) \leqslant H $ and $\pi_{1}^{wh}(G,e_G)$ is a topological group, Hence $H$ is an open subgroup of $\pi_{1}^{wh}(G,e_G)$.
\end{proof}

Pashaei et al. introduced a relative version of small loop transfer as follows.
\begin{definition}(\cite{Pasha1}).
Let $H\leq \pi_{1}(X,x_{0})$. A topological space $X$ is called an $H$-small loop transfer ($H$-SLT for short) space at $x_{0}$ if for every path $ \alpha $ in $X$ with $\alpha(0)=x_{0}$  and for every neighborhood $U$ of $x_{0}$ there is a neighborhood $V$ of $\alpha(1)=x$ such that for every loop $\beta$ in $V$ based at $x$ there is a loop $\gamma$ in $U$ based at $x_{0}$ such that  $[\alpha\ast\beta\ast\bar{\alpha}\ast\bar{\gamma}] \in H$.
\end{definition}

It is easy to see that every SLT space at $x_{0}$  is an $H$-SLT space at $x_{0}$, for any subgroup $H$ of $\pi_{1}(X,x_{0})$, therefore every topological group $G$ is an $H$-SLT space at $e_G$, for any subgroup $H$ of $\pi_{1}(G,e_G)$.

\begin{proposition}
Let $G$ connected and locally path connected topological group and $H \leqslant \pi_{1}(G,e_G)$. Then the following statment are equivalent.

(i) $H$ is an closed subgroup of $\pi_{1}^{qtop}(G,e_G)$.\

(ii) $H$ is an closed subgroup of $\pi_{1}^{wh}(G,e_G)$.\

(iii) $G$ is homotopically Hausdorff relative to $H$.\

(iv) $G$ is homotopically path Hausdorff relative to $H$.\

\end{proposition}

\begin{proof}
$(i) \Leftrightarrow (ii)$ follows  from corollary 2.6.\\
$(iii) \Leftrightarrow (iv)$ follows  from \cite[Theorem 2.5]{Pasha1} since $G$ is an $H$-SLT at $e_G$ and $\pi_{1}(G,e_G)$ is abelian, so $H$ is a normal subgroup of$\pi_{1}(G,e_G)$.\\
$(iv) \Leftrightarrow (i)$ follows  from \cite[Lemma 9]{BrazFa}.
\end{proof}

\begin{corollary}
A connected locally path connected topological group $G$ is homotopically Hausdorff if and only if $\pi_{1}^{qtop}(G,e_G)$ is a Hausdorff space.
\end{corollary}

\begin{proof}
Assume that $G$ is homotopically Hausdorff. So $G$ is homotopically Hausdorff relative to the trivial subgroup $H = \lbrace 1\rbrace$. Hence by Proposition 2.9 $\{ e_G \} $ is closed in $\pi_{1}^{qtop}(G,e_G)$. Therefore for every $g \in G$, $\{ g \} $ is closed in $\pi_{1}^{qtop}(G,e_G)$ since $\pi_{1}^{qtop}(G,e_G)$ is a quasitopological group. Hence $\pi_{1}^{qtop}(G,e_G)$ is $T_0$, which implies that it is a Hausdorff space since $\pi_{1}^{qtop}(G,e_G)$ is a topological group. The converse is trivial.
\end{proof}

\begin{theorem}
A topological group $G$ is a strong SLT space if $G$ is an abelian group or a path connected space.
\end{theorem}

\begin{proof}
Let $G$ be an abelian topological group and $a \in G$. We show that $G$ is a strong SLT space at $a$. For this let $U$ be a neighborhood of $a$ in $G$ and $b \in G$. We show that for every loop $\beta$ based at $b$ in the neighborhood $ba^{-1}U = \{ ba^{-1}u | u \in U \}$ of $b$ and every path $ \alpha $ in $G$ with $\alpha(0)=a , \alpha(1) = b$, there is a loop $\gamma$ in $U$ based at $a$  which is homotopic to $ \alpha\ast\beta\ast\bar{\alpha}$ relative to $\dot{I}$. Let $f$ be a loop in $G$ based at $a$ such that
\begin{displaymath}
f(t)= \left\{
\begin{array}{lr}
\alpha(3t)                     &       0\leq t\leq 1/3 \\
b     &      1/3\leq t\leq 2/3\\
\bar{\alpha}(3t-2)              &  2/3\leq t\leq 1.
\end{array}
\right.
\end{displaymath}
Also let $g$ be a loop in $G$ based at $a$ such that
\begin{displaymath}
g(t)= \left\{
\begin{array}{lr}
a                    &       0\leq t\leq 1/3 \\
ab^{-1}\beta(3t-1)     &      1/3\leq t\leq 2/3\\
a             &  2/3\leq t\leq 1.
\end{array}
\right.
\end{displaymath}
Therefore by Lemma 2.2 we have $[f^a][^ag] = [f.g]$. Note that
\begin{displaymath}
(f.g)(t)= \left\{
\begin{array}{lr}
\alpha^a(3t)                     &       0\leq t\leq 1/3 \\
bab^{-1}\beta(3t-1)     &      1/3\leq t\leq 2/3\\
\bar{\alpha}^a(3t-2)             &  2/3\leq t\leq 1.
\end{array}
\right.
\end{displaymath}
Since $G$ is abelian, hence $bab^{-1} = a$, $ \alpha ^a = ^a\alpha $ and $ \bar{\alpha} ^a = ^a\bar{\alpha}$. Therefore 
$$f.g = (^a\alpha)\ast(^a\beta)\ast (^a\bar{\alpha}) = ^a(\alpha\ast\beta\ast\bar{\alpha}).$$
Since $G$ is abelian, so $ [^af] = [f^a]$. Hence
$$ [^a(f\ast g)] = [(^af)\ast (^ag)] = [^af] [^ag] = [f^a][^ag] = [f.g] = [^a(\alpha\ast\beta\ast\bar{\alpha})]. $$
Therefore $[f \ast g] = [(\alpha\ast\beta\ast\bar{\alpha})]$.
If $\gamma = ^{ab^{-1}}\beta$, then $\gamma$ is a loop in $U$ based at $a$ since $\beta$ is a loop in $ba^{-1}U$ based at $b$. Since $[f] = [\alpha * \bar{\alpha}]$, we have
$$[\gamma] = [C_a][\gamma]= [\alpha * \bar{\alpha}][\gamma] = [f][g] = [(\alpha\ast\beta\ast\bar{\alpha})] $$
Hence $G$ is a strong SLT space at $a$.\\
 Now let $G$ be a path connected topological group and $a \in G$. We show that $G$ is a strong SLT space at $a$. For this let $U$ be a neighborhood of $a$ in $G$ and $b \in G$. We show that for every loop $\beta$ based at $b$ in the neighborhood $ba^{-1}U = \{ ba^{-1}u | u \in U \}$ of $b$ and every path $ \alpha $ in $G$ with $\alpha(0)=a , \alpha(1) = b$, there is a loop $\gamma$ in $U$ based at $a$  which is homotopic to $ \alpha\ast\beta\ast\bar{\alpha}$ relative to $\dot{I}$. Since $G$ is path connected so there is a path $\lambda$ in $G$ from $e_G$ to $a$. By proof of Theorem 2.4, we have
$$[\lambda\ast\alpha\ast\beta\ast\bar{\alpha}\ast\bar{\lambda}] = [(\lambda\ast\alpha)\ast\beta\ast\overline{(\lambda\ast\alpha)}]
= [(\lambda\ast\alpha)\ast\overline{(\lambda\ast\alpha)}][^{b^{-1}}\beta] = [C_{e_G}][^{b^{-1}}\beta] = [^{b^{-1}}\beta]. $$
Also $$[\lambda\ast(^{ab^{-1}}\beta)\ast\bar{\lambda}]  
= [\lambda\ast\bar\lambda][^{b^{-1}}\beta] = [C_{e_G}][^{b^{-1}}\beta] = [^{b^{-1}}\beta]. $$
Therefore $[\lambda\ast\alpha\ast\beta\ast\bar{\alpha}\ast\bar{\lambda}] = [\lambda\ast(^{ab^{-1}}\beta)\ast\bar{\lambda}]$, which implies that
$[\alpha\ast\beta\ast\bar{\alpha}] = [^{ab^{-1}}\beta]$. If $\gamma = ^{ab^{-1}}\beta$, then $\gamma$ is a loop in $U$ based at $a$ and $[\alpha\ast\beta\ast\bar{\alpha}] = [\gamma]$. Hence $G$ is a strong SLT space at $a$.

\end{proof}

\begin{corollary}
A topological group $G$ is an SLT space if $G$ is an abelian group or a path connected space.
\end{corollary}

 Brodskiy et al. \cite{BroU} proved that if a path connected space $X$ is a strong SLT space, then the fundamental group of $X$ with the whisker topology is equal to the fundamental group of $X$ with the lasso topology, which is always a topological group \cite{BroU1}. Therefore by Theorem 2.11 we have the following corollary.
\begin{corollary}
Let $G$ be a path connected topological group, then $ \pi_{1}^{wh}(G,e_G) $ is  a topological group.
\end{corollary}

\section{Covering theory of topological groups}

\begin{definition}\cite[Definition 8.2]{L}
Let $G$ be a topological group. By a covering group of $G$, we mean a pair $( \widetilde{G}, p)$ composed of a topological group $\widetilde{G}$  and of a homomorphism $p$  of $\widetilde{G}$ into $G$ such that $( \widetilde{G}, p)$ is a covering space of $G$.
\end{definition}

\begin{example}
Every covering of $S^{1}$ is a covering group.

\end{example}

For a topological group $G$, the category of covering  groups $CTG(G)$ is the category whose
\\ $\maltese$ \ \ objects are covering  groups $( \widetilde{G}, p)$;
  \\  $\maltese$ \ \ morphism for two objects $p :\widetilde{G} \to G$ and $q :\widetilde{H} \to G$ is  a continuous homomorphism function $\varphi :\widetilde{G} \to \widetilde{H}$ such that $q \circ \varphi =p$.
\begin{definition}

Two covering  groups  $p :\widetilde{G} \to G$ and $q :\widetilde{H} \to G$  are equivalent if there exists a homomorphism  $\varphi :\widetilde{G} \to \widetilde{H}$ such that $\varphi$ is homeomorphism and $q \circ \varphi =p$.
\end{definition}

 Torabi et al. \cite[Theorem 3.7]{t} showed that for a connected, locally path connected space $X$, there
is a one to one correspondence between its equivalent classes of connected covering spaces and the conjugacy classes of those subgroups of fundamental group $\pi_1(X,x)$ which contain an open normal subgroup of $\pi_1^{qtop}(X,x)$. If $G$ is a topological group, $\pi_1(G, e_G)$ is abelian and therefore every subgroup of $\pi_1(G, e_G)$ is normal. So  every  open subgroup of $\pi_1(G, e_G)$ has an open normal subgroup. Therefore for a connected, locally path connected topological group $G$, there is a one to one correspondence between its equivalent classes of connected covering spaces and the open subgroups of $\pi_1^{qtop}(X,x)$. In this section we introduce a similar classification for covering groups of $G$.

\begin{theorem}\label{Lifting Criterion}(Lifting Criterion for topological Group).
 Let $\widetilde{G} $ and $ G$ be topological groups and Let $H $ be a connected and locally path connected topological group, and let homomorphism $f: (H, e_H) \to (G, e_G)$ be continuous. If $p :(\widetilde{G}, e_{\widetilde{G}}) \to (G, e_G)$ is a covering  group, then there exists a unique continuous homomorphism $\widetilde{f}: (H, e_H) \to (\widetilde{G}, e_{\widetilde{G}})$ (where $e_{\widetilde{G}} \in p^{-1} (e_G)$) lifting
$f$ (i.e $p(\widetilde{f}) = f$) if and only if $f_* (\pi_1 (H, e_H)) \subseteq  p_* (\pi_1 (\widetilde{G}, e_{\widetilde{G}}))$.
\end{theorem}

\begin{proof}
If there exists $\widetilde{f}$ such that $ p \circ \widetilde{f} = f$, then $f_* (\pi_1 (H, e_H)) \subseteq  p_* (\pi_1 (\widetilde{G}, e_{\widetilde{G}}))$.
Conversely, let $h \in H$ and let $\lambda: I \to H$ be a path from $e_H$ to $h$; thus $f \circ \lambda$ is a path from $f(e_H) = e_G$ to $f(h)$. Since $p$ is a cover, there is a unique path $\widetilde{\lambda}$ in $\widetilde{G}$ that lifts $f \circ \lambda$ with $\widetilde{\lambda}(0) = e_{\widetilde{G}}$.
We define $\widetilde{f}$  by $\widetilde{f}(h) = \widetilde{\lambda}(1)$. By \cite[Theorem 10.13]{r}, $\widetilde{f}$ is unique well-defined  continuous map. So it is enough to show that $\widetilde{f}$ is homomorphism. For this, consider $h_1, h_2 \in H$ we must show that
 $$\widetilde{f}(h_1h_2)= \widetilde{f}(h_1)\widetilde{f}(h_2).$$ Let $\lambda_i: I \to H $ be a path from $e_H$ to $h_i$ where $e_H$ is the identity element of $H$, thus $f \circ \lambda_i$ is a path from $f(e_H) = e_G$ to $f(h_i)$ where $i=1,2$ and  $e_G$ is the the identity element of $G$.
 By Lemma \ref{path}, there exists the path $\lambda_1.\lambda_2$  in $H$ with starting point $e_H$ and end point $h_1h_2$ and $f \circ (\lambda_1.\lambda_2)$ is a path in $G$ with starting point $f(e_H)= e_G$ and end point $f(h_1h_2)$. Since $f$ is homomorphism, $f(h_1h_2)= f(h_1)f(h_2)$.
 Since $p$ is a cover, there are  unique paths $\widetilde{\lambda_i}$ in $\widetilde{G}$ that lifts $f \circ \lambda_i$ such that $\widetilde{\lambda_i}(0) = e_{\widetilde{G}}$ and $\widetilde{\lambda_i}(1) = \widetilde{f}(h_i)$ where $i=1,2$. We show that the path $ \widetilde{\lambda_1}.\widetilde{\lambda_2}$ is lifting of $f \circ (\lambda_1.\lambda_2)$ with starting point $e_{\widetilde{G}}$ and end point $\widetilde{f}(h_1)\widetilde{f}(h_2)$. For this, consider for every $t \in I$, $$ (p \circ (\widetilde{\lambda_1}.\widetilde{\lambda_2}))(t)=p \circ (\widetilde{\lambda_1}.\widetilde{\lambda_2}(t))= p \circ (\widetilde{\lambda_1}(t)\widetilde{\lambda_2}(t))=  p( (\widetilde{\lambda_1}(t)\widetilde{\lambda_2}(t)))= $$$$=p(\widetilde{\lambda_1}(t))p  (\widetilde{\lambda_2}(t))= (f \circ \lambda_1)(t) (f \circ \lambda_2)(t)=f ( \lambda_1(t)) (f  \lambda_2(t))= $$$$= f ( \lambda_1(t) \lambda_2(t))=f( ( \lambda_1. \lambda_2)(t))=f \circ (\lambda_1.\lambda_2(t))=(f \circ (\lambda_1.\lambda_2))(t).$$ Since $\widetilde{f}$ is well-defined, the end point of  $\widetilde{\lambda_1}.\widetilde{\lambda_2}$ is $\widetilde{f}(h_1h_2)$. Thus  $\widetilde{f}(h_1h_2)= \widetilde{f}(h_1)\widetilde{f}(h_2).$

\end{proof}

\begin{corollary}\label{equivalent}

Let $G$ be locally path connected topological group, and let $e_G$ be the identity element of $G$ . Let $( \widetilde{G}, p)$ and
$( \widetilde{H}, q)$ be covering  groups of $G$, and let $e_ {\widetilde{G}}\in p^{-1 }(e_G) = ker(p)$ and  $e_{\widetilde{H}} \in q^{-1 }(e_G) = ker(q)$. Then $p :\widetilde{G} \to G$
and $q :\widetilde{H} \to G$ are equivalent if and only if $q_*(\pi_1(\widetilde{H},e_{\widetilde{H}})) = p_*(\pi_1(\widetilde{G},e_ {\widetilde{G}})).$

\end{corollary}

\begin{proof}
Assume that $p :\widetilde{G} \to G$ and $q :\widetilde{H} \to G$  are equivalent, and let $\varphi: \widetilde{G} \to \widetilde{H}$ be
a homeomorphism such that $\varphi$ is homomorphism and $q \circ \varphi =p$. Then By \cite[Theorem 10.20]{r}, $q_*(\pi_1(\widetilde{H}, e_{\widetilde{H}})) $ and $p_*(\pi_1(\widetilde{G},e_ {\widetilde{G}}))$ are conjugate subgroups of $\pi_1(G, e_G)$. Since $\pi_1(G, e_G)$ is abelian, $q_*(\pi_1(\widetilde{H}, e_{\widetilde{H}})) = p_*(\pi_1(\widetilde{G},e_ {\widetilde{G}}))$.
Conversely, assume that  $q_*(\pi_1(\widetilde{H}, e_{\widetilde{H}})) = p_*(\pi_1(\widetilde{G},e_ {\widetilde{G}}))$. By Theorem \ref{Lifting Criterion}, there exists continuous homomorphism $\varphi: \widetilde{G} \to \widetilde{H}$  such that $q \circ \varphi =p$ and also there exists continuous homomorphism $\phi: \widetilde{H} \to \widetilde{G}$  such that $p \circ \phi =q$. Therefore $p :\widetilde{G} \to G$  and $q :\widetilde{H} \to G$ are equivalent.
\end{proof}

We recall from \cite{t} that the spanier group $\pi(\mathcal{U}, e_G)$ with respect to the open cover $\mathcal{U}=\lbrace U_{i} \ \vert \ i\in{I}\rbrace$ of $G$ is defined to be the subgroup of $\pi_{1}(G,e_G)$ which contains all homotopy classes having representatives of the following type:

$$\prod_{j=1}^{n}\alpha_{j}\ast\beta_{j}\ast\bar{\alpha}_{j}$$

 where $\alpha_{j}$'s are arbitrary path starting at $e_G$ and each $\beta_{j}$ is a loop inside of the open sets $U_{j}\in{\mathcal{U}}$. We generalize Theorem 10.42 of \cite{r} as follows.

\begin{theorem} \label{exist}
Let $G$ be connected locally path-connected topological group, and let $e_G$ be the identity element of $G$. If $H$ be a subgroup of $\pi_1(G, e_G)$, then there exists a covering  group $p :\widetilde{G} \to G$ such that $p_*(\pi_1(\widetilde{G},e_{\widetilde{G}}) )=H$ if
and only if there is an open covering $\mathcal{U}$ of $G$ such that $\pi (\mathcal{U}, e_G)\leq H$.

\end{theorem}

\begin{proof}
If there is a covering  map $p :\widetilde{G} \to G$ such that $p_*(\pi_1(\widetilde{G},e_{\widetilde{G}}) )=H$,then there is an open covering $\mathcal{U}$ of $G$ such that $\pi (\mathcal{U}, e_G)\leq H$ ( see\cite[Lemma 2.11]{s}) and we know every covering group is a covering map.
Conversely, let $P(G, e_G)$ be the family of all paths $f$ in $G$ with $f(0) = e_G$. We Define $f \sim g$ by $f(1)=  g(1)$ and $ [f * \bar{g}] \in H$. The relation $f \sim g$ is an equivalence relation on $P(G, e_G)$ and equivalent class of $f$ will be denoted by $<f>_H$. Let $\widetilde{G}_H $ be the set of equivalence classes. By \cite[Theorem 2.13]{s}, there exists  a covering map $p_H :\widetilde{G}_H \to G$ such that $ p_H(<f>_H) = f(1)$ and $(p_H)_*(\pi_{1}(\widetilde{G}_{H},e_{\widetilde{G}_{H}} ))=H$. So it is enough to show that $p_H$ is homomorphism. For this, we define
$$<f>_H<g>_H = <f.g>_H.$$
We show that this multiplication is well-defined and $\widetilde{G}_{H}$ is a topological group. Consider $<f>_H= <f'>_H$ and $<g>_H = <g'>_H$, so there exists $ h_1, h_2 \in H$ such that $ [f * \bar{f'}]= [h_1] $ and $[g * \bar{g'}]=[h_2]$ and so there exists $F_1:I \times I \to G$ such that $F_1(t,0)= f * \bar{f'}(t)$ and $ F_1(t,1)=h_1(t)$ and there exists $F_2:I \times I \to G$ such that $F_2(t,0)= g * \bar{g'}(t)$ and $ F_2(t,1)=h_2(t)$. We define $F:= I \times I \to G \times G \to G$ by $ F(t,s)= m(F_1(t,s), F_2(t,s))$ where $m:G \times G \to G$ is the multiplication map of $G$ and $$F(t,0)= m((f * \bar{f'})(t),(g * \bar{g'})(t))= (f * \bar{f'})(t)(g * \bar{g'})(t)= ((f.g)*(\bar{f'}.\bar{g'}))(t)$$ and $ F(t,1)=(h_1.h_2)(t)$. Thus $[(f.g)*(\bar{f'}.\bar{g'})]= [h_1.h_2]$ and by Lemma \ref{multi}, $[h_1.h_2]=[h_1][h_2]$ . Since $ h_1, h_2 \in H$ and $H$ is a subgroup of $\pi_1(G, e_G)$, $$[(f.g)*(\bar{f'}.\bar{g'})]= [h_1.h_2]= [h_1][h_2] \in H.$$
Thus the  multiplication of  $\widetilde{G}_{H}$ is well-defined.
 This multiplication is associative since the operation of $G$ is associative.
 The identity element of $\widetilde{G}_{H}$ is $C_{e_G}:I \to G$ by $C_{e_G}(t)=e_G$. By the definition of multiplication of  $\widetilde{G}_{H}$, for every element $ <f>_H \in \widetilde{G}_{H}$, $<f.C_{e_G}>_H=<f>_H=< C_{e_G}.f>_H$. Also for every element $ <f>_H \in \widetilde{G}_{H}$, we define the inverse element of $<f>_H$ by $f^{-1}:I \to G$ by $ f^{-1}(t)= (f(t))^{-1}$ such that $<f.f^{-1}>_H=<C_{e_G}>_H=< f^{-1}
 .f>_H$. Thus $\widetilde{G}_{H}$ is a group.
 Now we show that $p_H$ is homomorphism.  $p_H( <f.g>_H)= p_H(<f>_H)p_H(<g>_H)$ since $p_H(<f>_H)= f(1)$, $p_H(<g>_H)= g(1)$ and $p_H( <f.g>_H)= f.g(1)= f(1)g(1)$. $\widetilde{G}_{H}$ is a topological group. $ \theta: \widetilde{G}_{H} \times \widetilde{G}_{H}\to G$ by $\theta(<f>_H,<g>_H)= <f.g>_H$ is the multiplication map of $\widetilde{G}_{H}$. we show that $\theta$ is continuous, consider $<f.g,U>_H$ is an open neighborhood of $<f.g>_H$, $f(1)g(1) \in U$ and $G$ is a topological group so there exist open neighborhoods $U_1$ of $f(1)$ and $U_2$ of $f(2)$ in $G$ such that $m(U_1,U_2)\subseteq U$. Thus there exists the open neighborhoods $<f,U_1>_H$ of $<f>_H$ and $<g,U_2>_H$ of $<g>_H$ such that $\theta ( <f,U_1>_H,g,U_2>_H) \subseteq <f.g,U>_H$. $\lambda:\widetilde{G}_{H}\to \widetilde{G}_{H}$ by $\lambda (<f>_H)= <f^{-1}>_H $ is the inverse map of $\widetilde{G}_{H}$. we show that $\lambda$ is continuous, consider $<f^{-1},U>_H$ is an open neighborhood of $<f^{-1}>_H$ and $G$ is a topological group so $U^{-1}$ is  an open neighborhood in $G$. Thus $<f, U^{-1}>_H $ is  an open neighborhood of  $<f>_H$ and $ \lambda(<f, U^{-1}>_H)\subseteq <f^{-1},U>_H$.  Therefore $\widetilde{G}_{H}$ is a topological group.
\end{proof}

Since a topological group $G$ is a strong SLT at $e_G$, so by definition of strong SLT, for any neighborhood $U$ of $e_{G}$ in topological group $G$, there is an open covering $\mathcal{U}$ such that $\pi (\mathcal{U}, e_G)\leq i_{*}\pi_{1}(U,e_G) $. Therefore we conclude the following Corollary.
\begin{corollary} 
Let $G$ be connected locally path-connected topological group, and let $e_G$ be the identity element of $G$. If $H$ be a subgroup of $\pi_1(G, e_G)$, then there exists a covering  group $p :\widetilde{G} \to G$ such that $p_*(\pi_1(\widetilde{G},e_{\widetilde{G}}) )=H$ if
and only if there is a neighborhood $U$ of $e_{G}$ s.t. $i_{*}\pi_{1}(U,e_G) \leqslant H $.
\end{corollary}

By Corollary 3.5, Corollary3.7 and  Proposition 2.7 we conclude the following classification of covering groups.

\begin{corollary}
For a connected, locally path connected topological group $G$, there is a one to one correspondence between its equivalent classes of connected covering groups and the open subgroups of $\pi_1^{qtop}(X,x)$.
\end{corollary}

\begin{theorem}
Every covering space $\widetilde{G}$ of a connected locally path connected topological group $G$ is a topological group. Furthermore the covering map $q :\widetilde{G} \to G$ is homomorphism.
\end{theorem}

\begin{proof}
Let $G$ be a topological group and $( \widetilde{G}, q)$ is a covering space of $G$. If $q_*(\pi_1(\widetilde{G}, e_{\widetilde{G}})=H$ , then by Theorem \ref{exist} there exist covering group $( \widetilde{G}_H, p_H)$ of $G$. The covering spaces $( \widetilde{G}, q)$ and $( \widetilde{G}_H, p_H)$ are equivalent. So there exists homeomorphism $\varphi: \widetilde{G} \to \widetilde{G}_H$ such that $p_H \circ \varphi =q$. Now we define multiplication map $\psi: \widetilde{G} \times \widetilde{G} \to \widetilde{G}$ by $\psi(t,s)=\varphi^{-1}(\varphi(t).\varphi(s))$. Since $\psi$ is homeomorphism,  $p_H \circ \varphi =q$ and $\widetilde{G}_H$  is a topological group, $ \widetilde{G}$ is a topological group with continuous multiplication map $\psi$ and $q$ is homomorphism.
\end{proof}








\

\

\
\\
\\




\end{document}